\documentclass[10pt]{amsart}

\usepackage{amsmath}
\usepackage{amssymb}
\usepackage{bm}
\usepackage{graphicx}
\usepackage{psfrag}
\usepackage{color}
\usepackage{hyperref}
\hypersetup{colorlinks=true, linkcolor=blue, citecolor=magenta, urlcolor=wine}
\usepackage{url}
\usepackage{algpseudocode}
\usepackage{fancyhdr}
\usepackage{mathtools}
\usepackage{tikz-cd}
\usepackage{xy}
\input xy
\xyoption{all}
\usepackage{stmaryrd}
\usepackage{calrsfs}

\newtheorem*{maintheorem*}{Main Theorem}
\newtheorem{theorem}{Theorem}[section]
\newtheorem{prop}[theorem]{Proposition}

\newtheorem{lem}[theorem]{Lemma}
\newtheorem{cor}[theorem]{Corollary}

\theoremstyle{definition}
\newtheorem{defn}[theorem]{Definition}
\newtheorem{rem}[theorem]{Remark}
\newtheorem{ex}[theorem]{Example}

\numberwithin{equation}{section}

\newcommand{\cc}{\mathbb{C}}
\newcommand{\ff}{\mathbb{F}}
\newcommand{\nn}{\mathbb{N}}
\newcommand{\pp}{\mathbb{P}}
\newcommand{\qq}{\mathbb{Q}}
\newcommand{\rr}{\mathbb{R}}
\newcommand{\zz}{\mathbb{Z}}

\providecommand\ldb{\llbracket}
\providecommand\rdb{\rrbracket}

\newcommand{\gp}{\text{gp}}

\newcommand{\uu}{\mathcal{U}}



\begin{document}
	\mbox{}
	\title{Riemann zeta functions for Krull monoids}
	
	\author{Felix Gotti}
	\address{Department of Mathematics\\MIT\\Cambridge, MA 02139\\USA}
	\email{fgotti@mit.edu}
	
	\author{Ulrich Krause}
	\address{Department of Mathematics and Computer Science\\University of Bremen\\D-28359 Bremen\\Germany}
	\email{uli.krause@uni-bremen.de}

\date{\today}

\begin{abstract}
	The primary purpose of this paper is to generalize the classical Riemann zeta function to the setting of Krull monoids with torsion class groups. We provide a first study of the same generalization by extending Euler's classical product formula to the more general scenario of Krull monoids with torsion class groups. In doing so, the Decay Theorem is fundamental as it allows us to use strong atoms instead of primes to obtain a weaker version of the Fundamental Theorem of Arithmetic in the more general setting of Krull monoids with torsion class groups. Several related examples are exhibited throughout the paper, in particular, algebraic number fields for which the generalized Riemann zeta function specializes to the Dedekind zeta function.
\end{abstract}

\bigskip
\maketitle

\bigskip
\section{Introduction}
\label{sec:intro}

In the setting of natural numbers, primes are defined purely by the multiplicative operation. However, most proofs of the infinitude of prime numbers also invoke the additive operation. An exception is a famous proof by Euler which can be seen as essentially multiplicative in that it uses instead of addition just certain ``scaling" of numbers. This notion of scaling can be formalized as a function, which we call a \emph{scale}, on the set $\nn$ of positive integers or, more generally, on a multiplicative monoid. This naturally raises a related and more general question: which standard prime-related results can be established for a purely multiplicative monoid endowed with a scale? Especially, what can be said about the infinitude of primes? This paper addresses this and related questions by extending Euler's analysis (that is, the representation of the Riemann zeta function as Euler's product over primes) from the monoid of natural numbers to more general multiplicative monoids equipped with scales. The role of primes is played by the more comprising notion of strong atoms. To make sure there are enough of these atoms, the monoids we deal with here are assumed to be Krull monoids with torsion class groups. The Decay Theorem then assures that each nonunit element has some power that can be factored into strong atoms in a unique way, yielding an extension of the Fundamental Theorem of Arithmetic that we will be using as our main tool.
\smallskip

Thus, we will use the notion of strong atoms as well as the Decay Theorem to provide a generalization of the Riemann zeta function to the setting of Krull monoids with torsion class group, and then we extend Euler's product formula accordingly. The classical (real) Riemann zeta function $\zeta \colon \rr_{> 1} \to \rr$ is defined as $\zeta(s) = \sum_{n=1}^\infty \frac 1{n^s}$ for each $s \in \nn$, and Euler's product formula can be written as follows:
\begin{equation} \label{eq:Euler's product formula}
	\zeta(s) = \sum_{n=1}^\infty \frac 1{n^s} = \prod_{p \in \pp} \frac{1}{1 - \frac{1}{p^s}} \quad \text{for every} \quad s \in \nn,
\end{equation}
where $\pp$ denotes the set of prime numbers (inside the multiplicative monoid $\nn$). The representation in~\eqref{eq:Euler's product formula} is specially useful in analytic number theory. Unique factorization monoids are cases of particular importance, and they include the multiplicative monoid $\nn$. In order to extend the classical Riemann zeta function to a monoid $M$ more general than $\nn$ (namely, a Krull monoid $M$ with torsion class group), here we introduce the notion of a scale, which is a monoid homomorphism $\sigma \colon M \to \rr^\times$. For instance, for each $s \in \nn$, the monoid homomorphism $\sigma_s \colon \nn \to \rr^\times$ defined by $\sigma_s \colon n \mapsto n^s$ is a scale on the Krull monoid $\nn$ (in the definition of a scale~$\sigma$, one can actually replace $\rr$ by any other Archimedean field).
\smallskip

Abstract analytic number theory in the context of Krull monoids has been considered in the literature for about three decades. The development of such theory seems to be initiated by Geroldinger and Kaczorowski in~\cite{GK92}, and since then it has been further investigated in several papers (for instance, see the recent papers~\cite{jK17,mR22} and references therein). However, in most of these papers, including the ones just cited, the authors put the emphasis on Krull monoids with finite class groups. In the present paper, we consider the more general setting of Krull monoids with torsion class groups. The key observation, and one of our fundamental motivations, is that Krull monoids with torsion class groups have enough of certain atoms that are sufficiently strong to play the role of primes and so, after some reasonable restriction, we still have an alternative version of the Fundamental Theorem of Arithmetic in the more general class of Krull monoids with torsion class groups.
\smallskip

In a cancellative and commutative monoid, a strong atom\footnote{Strong atoms are also called absolute irreducibles in the literature.} is an atom whose powers can be factored into atoms in a unique way (the obvious way). In contrast to the multiplicative monoid $\nn$, many Krull monoids do not satisfy the unique factorization property, and so although Krull monoids are always atomic, they do not contain enough primes to guarantee factorizations into primes. However, in light of the Decay Theorem, we know that in a given Krull monoid with torsion class group, the set consisting of all strong atoms is large enough such that the submonoid generated by this set is a unique factorization monoid whose radical monoid is the whole given Krull monoid. In this paper, we harness this fact and, after replacing the role of primes by that of strong atoms, we obtain a suitable generalization of the classical Riemann zeta function. Given the importance of the class of Krull monoids in factorization theory and abstract divisibility theory, we hope that the generalized versions of both the classical Riemann zeta function and the corresponding Euler's product formula we propose here eventually become handy tools at the disposal of interested authors to further study the analytic and arithmetic theory of Krull monoids with torsion class groups.
\smallskip

In Section~\ref{sec:background}, we briefly mention some fundamental notation, definitions, and well-known results we are using throughout this paper. 
\smallskip

In Section~\ref{sec:ST and KM}, we discuss sets of strong atoms for general monoids, putting special emphasis on Krull monoids with torsion class groups. We begin by introducing the notion of strong atoms, and we provide various examples to compare this notion with those of an atom and a prime. Then we discuss the set of strong atoms in the setting of Krull monoids, and we present the Decay Theorem for Krull monoids with torsion class groups, which guarantees that every element has a power that can be uniquely factored into strong atoms. We use the Decay Theorem to gain a better understanding of the atoms and strong atoms of Krull monoids with class group $\zz/2\zz$.
\smallskip

In Section~\ref{sec:Zeta function for KM and Euler Product}, we introduce and study the main objects of this paper, the generalized Riemann zeta functions: for a Krull monoid $M$ with torsion class group, we define a function $\zeta_M$ on the set of all scales $\sigma \colon M \to \rr^\times$ in such a way that $\zeta_{M}(\sigma)$ is an infinite sum, which specializes to the classical Riemann zeta series $\zeta(s)$ when the Krull monoid is $\nn$ and the scales are $\sigma_s \colon n \mapsto n^s$ for each $s \in \nn$. The function $\zeta_M$ on the set of all the scales on~$M$ is what we call the (generalized) Riemann zeta function of~$M$. As the primary result in the corresponding section, we extend Euler's classical product formula to Krull monoids with torsion class groups using the introduced and more general version of the Riemann zeta function: we provide a representation of $\zeta_{M}(\sigma)$ as a product in terms of strong atoms that generalizes Euler's classical product formula. We also highlight, as corollaries, the cases of unique factorization monoids and half-factorial monoids. In doing so, we present criteria for the infinitude of strong atoms in Krull monoids with torsion class groups. This applies specially to monoids satisfying the unique factorization property and, moreover, it explains in a certain sense the existence of infinitely many primes not only in the classical case of $\nn$ but also beyond.
\smallskip

Motivated by the arithmetic of the multiplicative monoid $\nn$ of natural numbers, we provide throughout Sections~\ref{sec:ST and KM} and~\ref{sec:Zeta function for KM and Euler Product} several more general examples, including non-factorial multiplicative submonoids of $\nn$ and rings of integers of algebraic number fields with their Dedekind zeta functions.

\bigskip
\section{Background}
\label{sec:background}

\smallskip
\subsection{General Notation} 

Following standard notation, we let $\zz$, $\qq$, $\rr$, and $\cc$ denote the set of integers, rational numbers, real numbers, and complex numbers, respectively. In addition, we let $\pp, \nn$, and~$\nn_0$ denote the set of primes, positive integers, and nonnegative integers, respectively. When we write $n = \prod_{p \in \pp} p^{n(p)}$ for some $n \in \nn$, we implicitly assume that this is the standard prime decomposition of $n$: the sequence $(n(p))_{p \in \pp}$ consists of nonnegative integers and almost all its terms equal zero. For $p \in \pp$ and $n \in \nn$, we let $\ff_{p^n}$ stand for the finite field of cardinality $p^n$. For $b,c \in \zz$ with $b \le c$, we let $\ldb b,c \rdb$ denote the discrete closed interval between~$b$ and~$c$, i.e.,
\[
	\ldb b,c \rdb = \{n \in \zz : b \le n \le c\}.
\]
Also, for $S \subseteq \rr$ and $r \in \rr$, we set $S_{\ge r} = \{s \in S : s \ge r\}$ and $S_{> r} = \{s \in S : s > r\}$. For every nonzero $q \in \qq$, we denote the unique $n \in \zz$ and $d \in \nn$ such that $q = n/d$ and $\gcd(n,d)=1$ by $\mathsf{n}(q)$ and $\mathsf{d}(q)$, respectively. For any nonempty set $I$, we let $\delta_{ij}$ denote the \emph{Kronecker's delta} on $I$, which means that for any indices $i,j \in I$, the equality $\delta_{ij} = 1$ holds when $i=j$ while the equality $\delta_{ij} = 0$ holds otherwise\footnote{Although the notation we use for the Kronecker's delta does not specify its domain of indices, this is the classical notation and hardly a source of ambiguity.}. Finally, for two disjoint sets $S$ and $T$, we often write their union as $S \sqcup T$ instead of $S \cup T$ to emphasize that they do not overlap.

\smallskip
\subsection{Commutative Monoids and Factorizations} 

A semigroup with an identity element is called a \emph{monoid}. For the rest of this section, monoids will be written multiplicatively. A commutative monoid~$M$ is said to be \emph{cancellative} provided that, for all $b,c,d \in M$, the equality $b \cdot c = b \cdot d$ implies that $c=d$. Throughout this paper, we will tacitly assume that every monoid we deal with is cancellative and commutative. We let $M^\bullet$ denote the set consisting of all elements of $M$ except the identity. The group consisting of all the units (i.e., invertible elements) of $M$ is denoted by $\uu(M)$. The quotient $M/\mathcal{U}(M)$ is a monoid, which is denoted by $M_{\text{red}}$. The monoid $M$ is called \emph{reduced} when $\uu(M)$ is the trivial group, in which case we identify $M_{\text{red}}$ with $M$. The \emph{quotient group} of $M$, here denoted by $\gp(M)$, is the unique abelian group (up to isomorphism) satisfying that any abelian group containing an isomorphic image of~$M$ also contains an isomorphic image of $\gp(M)$. The monoid $M$ is called \emph{torsion-free} if $\gp(M)$ is a torsion-free abelian group. The \emph{rank} of $M$, denoted by $\text{rank} \, M$, is the rank of $\gp(M)$ as a $\zz$-module, that is, the dimension of the vector space $\qq \otimes_\zz \gp(M)$ over~$\qq$. If~$S$ is a subset of $M$, then we let $\langle S \rangle$ denote the submonoid of $M$ generated by $S$. If $M = \langle S \rangle$ for some finite set $S$, then $M$ is said to be \emph{finitely generated}.
\smallskip

For $b,c \in M$, we say that $b$ \emph{divides} $c$ \emph{in} $M$ if there exists $b' \in M$ such that $c = b \cdot b'$, in which case we write $b \mid_M c$, dropping the subscript $M$ precisely when $M$ is the free multiplicative monoid~$\nn$. Two elements $b,c \in M$ are \emph{associates} in $M$ if $b \mid_M c$ and $c \mid_M b$. An element $p \in M \setminus \mathcal{U}(M)$ is called \emph{prime} (resp., \emph{primary}) if for any $b,c \in M$ the two relations $p \mid_M b \cdot c$ and $p \nmid_M b$ guarantee that $p \mid_M c$ (resp., $p \mid_M c^n$ for some $n \in \nn$). Clearly, every prime is a primary element. We let $\mathcal{P}(M)$ denote the set of all primes in $M$. A submonoid $N$ of $M$ is \emph{divisor-closed} if for each $b \in N$ and $d \in M$ the relation $d \mid_M b$ implies that $d \in N$. Let $S$ be a nonempty subset of $M$. An element $d \in M$ is a \emph{common divisor} of~$S$ provided that $d \mid_M s$ for all $s \in S$. A common divisor $d$ of $S$ is a \emph{greatest common divisor} of $S$ if $d$ is divisible by all common divisors of $S$. We let $\gcd_M(S)$ denote the set consisting of all the greatest common divisors of $S$ (we will drop the subscript $M$ from $\gcd_M(S)$ when we see no danger of ambiguity). A \emph{divisor theory} of $M$ is a monoid homomorphism $\tau \colon M \to F$, where~$F$ is a free commutative monoid, satisfying the following two conditions:
\begin{enumerate}
	\item if $b, c \in M$ and $\tau(b) \mid_F \tau(c)$, then $b \mid_M c$;
	\smallskip
	
	\item for every $d \in F$ there exist $b_1, \dots, b_n \in M$ with $d = \gcd\{\tau(b_1), \dots, \tau(b_n)\}$.
\end{enumerate}

An element $a \in M \setminus \uu(M)$ is called an \emph{atom} if for all $b,c \in M$ the equality $a = b \cdot c$ implies that either $b \in \uu(M)$ or $c \in \uu(M)$. The set consisting of all the atoms of $M$ is denoted by $\mathcal{A}(M)$. Observe that $\mathcal{P}(M) \subseteq \mathcal{A}(M)$. Following Anderson and Quintero~\cite{AQ97}, we say that $M$ is an \emph{AP-monoid} if $\mathcal{P}(M) = \mathcal{A}(M)$. A formal product $a_1 \cdots a_\ell$, where $a_1, \dots, a_\ell \in \mathcal{A}(M)$, whose actual product in~$M$ is $b$ is called a \emph{factorization} of~$b$. Two factorizations are identified up to order and associates; that is, factorizations can be formally considered as elements in the free commutative monoid on the set $\mathcal{A}(M_{\text{red}})$. For each $b \in M$, we let $\mathsf{Z}_M(b)$ denote the set of all the factorizations of $b$, dropping the subscript $M$ from $\mathsf{Z}_M(b)$ when we see no risk of ambiguity. Following Cohn~\cite{pC68}, we say that~$M$ is \emph{atomic} if every nonunit element of $M$ has at least one factorization. A factorization $z$ consisting of $\ell$ atoms (counting repetitions) is said to have \emph{length} $\ell$ and, when it is convenient, we denote the length~$\ell$ by~$|z|$. For $b \in M$, we set
\[
	\mathsf{L}_M(b) := \{|z| : z \in \mathsf{Z}(b)\},
\]
and we drop the subscript $M$ from $\mathsf{L}_M(b)$ when we see no risk of ambiguity. The monoid $M$ is called a \emph{unique factorization monoid} (resp., \emph{half-factorial monoid}) provided that $|\mathsf{Z}(b)| = 1$ (resp., $|\mathsf{L}(b)| = 1$) for every $b \in M$. For simplicity, it is customary to let the acronyms UFM and HFM stand for the terms `unique factorization monoid' and `half-factorial monoid', respectively.

\smallskip
\subsection{Convex Cones} 

Let $V$ be a finite-dimensional vector space over an ordered field $\ff$. A nonempty convex subset~$C$ of $V$ is called a \emph{cone} provided that $C$ is closed under linear combinations with nonnegative coefficients. In particular, any cone $C$ is an additive submonoid of $V$, and $C$ is called \emph{pointed} if it is reduced as an additive monoid. If $X$ is a nonempty subset of~$V$, then the \emph{conic hull} of~$X$ is defined as follows:
\[
	\text{cone}_V(X) := \bigg\{ \sum_{i=1}^n c_i x_i : n \in \nn \ \text{and} \ (c_i, x_i) \in \ff_{\ge 0} \times X  \ \text{for every} \ i \in \ldb 1,n \rdb \bigg\}.
\]
When we see no risk of ambiguity, we write $\text{cone}(X)$ instead of $\text{cone}_V(X)$. A cone in $V$ is called \emph{simplicial} if it is the conic hull of a linearly independent set of vectors, while it is called \emph{polyhedral} provided that it can be expressed as the intersection of finitely many closed half-spaces of~$V$. It is clear that every simplicial cone is polyhedral. Farkas-Minkowski-Weyl Theorem states that a cone is polyhedral if and only if it is the conic hull of a finite set~\cite[Section~1]{CC58}.
\smallskip

Let $C$ be a cone in $V$. A \emph{face} of~$C$ is a cone $F$ contained in $C$ satisfying the following condition: for all $x,y \in C$ the fact that the open line segment
\[
	\{tx + (1-t)y : t \in \ff \text{ and } 0 < t < 1\}
\]
intersects $F$ implies that both $x$ and $y$ belong to $F$. If $F$ is a face of $C$ and $F'$ is a face of $F$, then it is clear that $F'$ must be a face of $C$. Let $\langle \cdot, \cdot \rangle$ be a fixed inner product on $V$. For a nonzero vector $u \in V$, consider the hyperplane $H := \{x \in V : \langle x, u \rangle = 0 \}$, and denote the closed half-spaces
\[
	\{x \in V : \langle x, u \rangle \le 0 \} \quad \text{ and } \quad \{x \in V : \langle x, u \rangle \ge 0 \}
\]
by $H_u^-$ and $H_u^+$, respectively. If a cone~$C$ satisfies that $C \subseteq H_u^-$ (resp., $C \subseteq H_u^+$), then $H$ is called a \emph{supporting hyperplane} of~$C$ and~$H_u^-$ (resp., $H_u^+$) is called a \emph{supporting half-space} of $C$. A face $F$ of $C$ is called \emph{exposed} if there exists a supporting hyperplane $H$ of $C$ such that $F = C \cap H$.

\bigskip
\section{Strong Atoms and the Decay Theorem}
\label{sec:ST and KM}

In this section, we discuss the notion of strong atoms, putting special emphasis on the set of strong atoms of Krull monoids. We also discuss the Decay Theorem for Krull monoids with torsion class groups, which will be an important tool in Section~\ref{sec:Zeta function for KM and Euler Product}. The Decay Theorem roughly states that Krull monoids with torsion class groups contain enough strong atoms to factor a power of any given nonunit element. Since the objects of interest in this paper are reduced Krull monoids, and every reduced Krull monoid can be embedded into a free commutative monoid (see, for instance, \cite[Theorem~2.4.8]{GH06}), from this point on we tacitly assume that every monoid mentioned in the rest of this paper is torsion-free.

\medskip
\subsection{Strong Atoms}

For the rest of this section, let $M$ be a monoid. An atom $a \in \mathcal{A}(M)$ is called a \emph{strong atom} if for each $b \in M \setminus \mathcal{U}(M)$ the fact that $b \mid_M a^n$ for some $n \in \nn$ implies that $b$ is associate with $a^k$ in~$M$ for some $k \in \nn$; that is, $|\mathsf{Z}(a^n)| = 1$ for every $n \in \nn$. We let $\mathcal{S}(M)$ denote the set consisting of all the strong atoms of $M$. The inclusion $\mathcal{P}(M) \subseteq \mathcal{S}(M)$ can be readily checked, while the inclusion $\mathcal{S}(M) \subseteq \mathcal{A}(M)$ follows directly from the given definitions. When $M$ is an AP-monoid (and so a UFM because $M$ is assumed to be atomic),
\[
	\mathcal{P}(M) = \mathcal{S}(M) = \mathcal{A}(M).
\]
However, the inclusions $\mathcal{P}(M) \subseteq \mathcal{S}(M)$ and $\mathcal{S}(M) \subseteq \mathcal{A}(M)$ may be strict, as we will see in Examples~\ref{ex:strong atoms of rank-1 monoids}, \ref{ex:strong atom of a rank-3 monoid}, and~\ref{ex:strong atoms, a rank-2 example}. Also, in Example~\ref{ex:Hilbert monoid} we will determine the set of strong atoms of the Hilbert monoid $H$ (to be properly introduced later) and obtain that $\mathcal{P}(H) \subsetneq \mathcal{S}(H) \subsetneq \mathcal{A}(H)$. It turns out that the only rank-$1$ monoid containing strong atoms is $\nn_0$ (up to isomorphism), even though some rank-$1$ monoids may contain infinitely many non-associate atoms.

\begin{ex} \label{ex:strong atoms of rank-1 monoids}
	Let $M$ be an (additive) submonoid of $\qq$ that is atomic but not a group. We first claim that $M$ cannot contain a positive number and a negative number simultaneously. Suppose, by way of contradiction, that this is not the case, and fix $q \in M^\bullet$. Then we can take $r \in M$ such that $qr < 0$ and, after taking $m \in \nn$ such that $\mathsf{d}(q) \mathsf{n}(r)$ divides $-m-1$, we see that $n := (-m-1) \frac qr \in \nn$. Thus, $-q = mq + nr \in M$, which implies that $q \in \uu(M)$. However, the fact that every element of~$M$ is a unit contradicts that $M$ is not a group, and so our claim follows. Hence, after replacing~$M$ by its isomorphic copy $-M$  if necessary, we can assume that $M$ is an atomic submonoid of $\qq_{\ge 0}$. If $M = s_0 \nn_0$ for some $s_0 \in \qq_{> 0}$, then $M$ is a UFM with $\mathcal{A}(M) = \{s_0\}$. In this case, $s_0$ is a prime element of $M$ and, therefore, a strong atom. Assume, otherwise, that $M \neq s\nn_0$ for any $s \in \qq_{> 0}$. We proceed to show that, in this case, $M$ contains no strong atoms. Take $a \in \mathcal{A}(M)$. Since $M$ is atomic and $M \neq a \nn_0$, we can pick $b \in \mathcal{A}(M) \setminus a \nn_0$. Observe that $b \mid_M \mathsf{n}(a) \mathsf{n}(b) = (\mathsf{d}(a)\mathsf{n}(b))a$, and so the fact that $M$ is atomic guarantees the existence of a factorization of $(\mathsf{d}(a)\mathsf{n}(b))a$ in $M$ containing the atom~$b$. Hence~$a$ is not a strong atom of $M$ and, as a result, $\mathcal{S}(M)$ is empty. In particular, it is well known and not hard to verify that $M_0 := \big\langle \frac1p : p \in \pp \big\rangle$ is an additive submonoid of $\qq$ with $\mathcal{A}(M_0) = \big\{ \frac1p : p \in \pp \big\}$, and so $M_0$ is an example of an atomic monoid with infinitely many non-associate atoms containing no strong atoms and, therefore, satisfying that
	\[
		\mathcal{P}(M_0) = \mathcal{S}(M_0) \subsetneq \mathcal{A}(M_0).
	\]
	The atomic structure and the arithmetic of lengths of additive submonoids of $\qq$ have been actively studied during the past few years (see \cite{GGT21,fG19} and references therein).
\end{ex}

There are also finite-rank atomic monoids whose sets of atoms and strong atoms coincide and strictly contain their corresponding sets of primes. This observation is illustrated in the following example.

\begin{ex} \label{ex:strong atom of a rank-3 monoid}
	Let $M$ be the additive submonoid of the free abelian group $\zz^3$ generate by the set $A := \{(\pm 1, \pm 1, 1) \}$. It is clear that $M$ is a rank-$3$ atomic monoid with $\mathcal{A}(M) = A$. One can actually verify that every element of $A$ is a strong atom of $M$ (this is also an immediate consequence of Proposition~\ref{prop:strong atoms of finite-rank monoids} below). On the other hand, the equality
	\[
		(-1,-1,1) + (1,1,1) = (-1,1,1) + (1,-1,1)
	\]
	implies that $\mathcal{P}(M)$ is empty. Therefore
	\[
		\mathcal{P}(M) \subsetneq \mathcal{S}(M) = \mathcal{A}(M).
	\]
	It particular, from this example we can see that a monoid generated by strong atoms may not be a UFM (unlike the standard fact that a monoid generated by primes is a UFM). The atomic structure of additive submonoids of $\zz^n$ has been studied by several authors in the past few years (see the recent papers~\cite{fG20a,gL22,LRT23}).
\end{ex}

In the following example, we exhibit a finite-rank atomic monoid whose set of strong atoms strictly contains its set of primes and is strictly contained in its set of atoms.

\begin{ex} \label{ex:strong atoms, a rank-2 example}
	Consider the additive submonoid $M$ of the free abelian group $\zz^2$ generated by the elements $a_0 := (0,2)$, $a_1 := (1,1)$, and $a_2 := (2,0)$. One can easily see that $\mathcal{A}(M) = \{a_0, a_1, a_2\}$, and so $M$ is atomic. From the fact that $\nn_0 a_0$ is a UFM and a divisor-closed submonoid of $M$ with $\mathcal{A}(\nn_0 a_0) = \{a_0\}$, we infer that $a_0$ is a strong atom of~$M$. By a similar argument, we obtain that~$a_2$ is also a strong atom of~$M$ (the fact that $a_0$ and $a_2$ are strong atoms of $M$ is a special case of Proposition~\ref{prop:strong atoms of finite-rank monoids}). On the other hand, as $2a_1 = a_0 +  a_2$, the atom $a_1$ is not a strong atom of $M$. The same equality shows that the strong atoms $a_0$ and $a_2$ are not primes. Hence $\mathcal{P}(M)$ is empty while $\mathcal{S}(M) = \{a_0, a_2\}$ and, therefore,
	\[
		\mathcal{P}(M) \subsetneq \mathcal{S}(M) \subsetneq \mathcal{A}(M).
	\]
\end{ex}

The description of the sets of strong atoms in the monoids discussed in Examples~\ref{ex:strong atoms of rank-1 monoids}, \ref{ex:strong atom of a rank-3 monoid}, and~\ref{ex:strong atoms, a rank-2 example} can be deduced from the next proposition, which gives a geometric characterization of the set of strong atoms of a finite-rank monoid. Finitely generated submonoids of finite-rank free abelian groups are the so-called \emph{affine monoids}, which have been systematically investigated by many authors given the relevant role they play in combinatorial commutative algebra and affine toric geometry (see the graduate textbooks~\cite{CLS11,MS05} and references therein). Factorizations in submonoids of finite-rank free abelian groups have been recently studied in the literature (see, for instance, the papers \cite{GKL19,fG20}).

\begin{prop} \label{prop:strong atoms of finite-rank monoids}
	Let $M$ be a finite-rank atomic (additive) monoid, and let $V$ be the $\qq$-vector space $\qq \otimes_\zz \emph{gp}(M)$. Then an element $a \in M^\bullet$ is a strong atom of~$M$ if and only if $\qq_{\ge 0} a$ is a one-dimensional face of $\emph{cone}_V(M)$ such that $M \cap \qq_{\ge 0} a = \nn_0 a$.
\end{prop}

\begin{proof}
	Let $d$ be the rank of $M$, and set $C := \text{cone}_V(M)$. Since $M$ is cancellative and torsion-free, the natural map $M \to \gp(M) \to \qq \otimes_\zz \gp(M)$ induces an embedding of $M$ into $V$, and so we can identify~$M$ with its isomorphic copy inside $V$. Hence, after identifying $V$ with $\qq^d$, we can assume that $M$ is a submonoid of $\zz^d$.
	
	For the direct implication, suppose that $a$ is a strong atom of $M$. To argue that $\qq_{\ge 0} a$ is a one-dimensional face of $C$, take $x,y \in C$ such that the open segment
	\[
		\{tx + (1-t)y : t \in \qq \cap (0,1)\}
	\]
	intersects $\qq_{\ge 0} a$. Therefore we can pick $c,c' \in \nn$ such that $cx + c'y \in \nn a$. Since $x,y \in C$, we can take pairwise distinct nonzero elements $v_1, \dots, v_k \in M$ such that $x = \sum_{i=1}^k q_i v_i$ and $y = \sum_{i=1}^k q'_i v_i$ for some $q_1, \dots, q_k, q'_1, \dots, q'_k \in \qq_{\ge 0}$ such that $\max\{q_i, q'_i\} > 0$ for every $i \in \ldb 1,k \rdb$. Observe that \[
		\sum_{i=1}^k (cq_i + c'q_i')v_i = cx + c'y \in \nn a.
	\]
	As a result, one can take $c_1, \dots, c_k \in \nn$ such that $\sum_{i=1}^k c_i v_i \in \nn a$. Since $a$ is a strong atom and $M$ is atomic, it follows that $v_1, \dots, v_k \in \nn a$, whence $x \in \qq_{\ge 0} a$ and $y \in \qq_{\ge 0} a$. Consequently, $\qq_{\ge 0} a$ is a one-dimensional face of $C$.
	
	For the second part of the direct implication, set $N := M \cap \qq_{\ge 0} a$. Note that $N$ is reduced as it is a submonoid of $\qq_{\ge 0}a$. In addition, since $\qq_{\ge 0} a$ is a face of $C$, we see that $N$ is a divisor-closed submonoid of $M$. Hence $N$ is atomic, and $a$ is also a strong atom of $N$. Let $b$ be an atom of $N$, and write $b = q a$ for some $q \in \qq_{> 0}$. Observe that $b \mid_N \mathsf{n}(q) a$. As a consequence, the fact that $a \in \mathcal{S}(N)$ ensures that $b = a$. Thus, $a$ is the only atom of the reduced atomic monoid $N$, and this implies that $M \cap \qq_{\ge 0} a = N = \nn_0 a$.
	\smallskip
	
	For the converse, suppose that $\qq_{\ge 0} a$ is a one-dimensional face of $C$ with $M \cap \qq_{\ge 0} a = \nn_0 a$. Since $\qq_{\ge 0} a$ is a face of $C$, the monoid $\nn_0 a$ is a divisor-closed submonoid of~$M$ such that $\mathcal{A}(\nn_0 a) = \{a\}$. Thus, the fact that $\nn_0 a$ is a divisor-closed submonoid of $M$ implies that $a \in \mathcal{A}(M)$. To argue that $a$ is actually a strong atom of $M$, suppose that $\sum_{i=1}^k a_i \in \nn a$ for some $a_1, \dots, a_k \in \mathcal{A}(M)$. Fix $j \in \ldb 1,k \rdb$. After setting $b_j := \big(\sum_{i=1}^k a_i \big) - a_j$, we can write
	\[
		\frac 12 a_j + \frac 12 b_j = \frac m2 a \in \qq_{\ge 0} a
	\]
	for some $m \in \nn$. Since $\qq_{\ge 0} a$ is a face of $C$, we obtain that $a_j \in \qq_{\ge 0}a$, and so $a_j \in M \cap \qq_{\ge 0} a = \nn_0 a$. As $a_j \in \mathcal{A}(M)$, the fact that $\nn_0 a$ is a divisor-closed submonoid of $M$ guarantees that $a_j \in \mathcal{A}(\nn_0 a)$ and, therefore, $a_j = a$. Hence $a_i = a$ for every $i \in \ldb 1,k \rdb$. Thus, we conclude that every element of $\nn a$ has a unique factorization in $M$, which means that $a \in \mathcal{S}(M)$.
\end{proof}

\begin{cor} \label{cor:strong atoms of finite-rank monoids}
	Let $M$ be a finite-rank atomic (additive) monoid, and let $V$ be the $\qq$-vector space $\qq \otimes_\zz \emph{gp}(M)$. Then the following statements hold.
	\begin{enumerate}
		\item If $\emph{cone}_V(M)$ is polyhedral, then $|\mathcal{S}(M)| < \infty$.
		\smallskip
		
		\item If $\emph{cone}_V(M)$ is simplicial, then $|\mathcal{S}(M)| \le \emph{rank} \, M$.
	\end{enumerate}
\end{cor}

\begin{proof}
	(1) In light of Proposition~\ref{prop:strong atoms of finite-rank monoids}, each strong atom of $M$ belongs to a one-dimensional face of $C := \text{cone}_V(M)$. Since $C$ is polyhedral, it suffices to verify that every one-dimensional face of $C$ contains only finitely many strong atoms of $M$. This is the case as, according to Proposition~\ref{prop:strong atoms of finite-rank monoids}, if a one-dimensional face $F$ of $C$ contains a strong atom $a$ of $M$, then the divisor-closed submonoid $M \cap F$ of $M$, which contains $\mathcal{S}(M) \cap F$, equals $\nn_0 a$ and so contains precisely one atom of $M$.
	\smallskip
	
	(2) As in the previous part, set $C := \text{cone}_V(M)$. Since $M$ has finite rank, $V$ has finite dimension, and so there are finitely many linearly independent vectors $v_1, \dots, v_k$ with $k \le \dim V = \text{rank} \, M$ such that $C = \text{cone}(v_1, \dots, v_k)$. Because $v_1, \dots, v_k$ are linearly independent, the one-dimensional faces of~$C$ are $\qq_{\ge 0} v_1, \dots, \qq_{\ge 0} v_k$. Finally, it follows from Proposition~\ref{prop:strong atoms of finite-rank monoids} that all the strong atoms of $M$ belong to one-dimensional faces of $C$ and each one-dimensional face of $C$ contains at most one strong atom of~$M$ (argued in the previous paragraph). Hence $|\mathcal{S}(M)| \le k \le \text{rank} \, M$.
\end{proof}

\begin{rem}
	In both Proposition~\ref{prop:strong atoms of finite-rank monoids} and Corollary~\ref{cor:strong atoms of finite-rank monoids}, we have assumed that the monoids in the statements are atomic. There are submonoids of $\zz^n$ that are not atomic: for instance, the nonnegative cone of $\zz^2$ with respect to the lexicographical order is a submonoid of $\zz^2$ that is not atomic.
\end{rem}

We can use Proposition~\ref{prop:strong atoms of finite-rank monoids} to produce finite-rank atomic monoids with infinitely many strong atoms.

\begin{ex}
	For each $n \in \nn_0$, set $a_n :=  (n^2-1, 2n, n^2+1)$, and let $M$ be the additive submonoid of $\zz^3$ generated by the set of vectors $A := \{a_n : n \in \nn_0 \}$. Observe that, for each $n \in \nn_0$, the ray $\qq_{\ge 0} a_n$ intersects the plane $z=1$ at the upper-half circle determined by $x^2 + y^2 = 1$ and $y \ge 0$, and so the fact that such upper-half circle is concave down (in the plane determined by $z=1$) implies that each of the rays $\qq_{\ge 0} a_n$ is a one-dimensional face of $\text{cone}_{\qq^3}(M)$ satisfying that $M \cap \qq_{\ge 0} a_n = \nn_0 a_n$. Hence it follows from Proposition~\ref{prop:strong atoms of finite-rank monoids} that $A \subseteq \mathcal{S}(M)$. As a result, we conclude that $M$ is a rank-$3$ atomic monoid with $\mathcal{S}(M) = \mathcal{A}(M) = A$.
\end{ex}

Every monoid we have dealt with so far has finite rank. We conclude this subsection with two examples of infinite-rank monoids with infinitely many strong atoms. 

\begin{ex} \label{ex:free commutative monoid}
	Let $I$ be a nonempty set. A function $f \colon I \to \nn_0$ is said to have \emph{finite support} provided that $f(i) = 0$ for all but finitely many $i \in I$. Let $\nn_0^{(I)}$ denote the set consisting of all functions $f \colon I \to \nn_0$ having finite support. Under the standard pointwise addition of functions, $\nn_0^{(I)}$ is the free commutative monoid on the set $\{e_i : i \in I\}$, where for each $i \in I$ the function $e_i$ is the elementary function determined by the assignments $e_i \colon j \mapsto \delta_{ij}$ (for all $j \in I$). It is clear that $\gp(\nn_0^{(I)}) = \zz^{(I)}$, and so $\text{rank} \, \nn_0^{(I)} = |I|$. In addition, the fact that $\nn_0^{(I)}$ is the free commutative monoid on $I$ guarantees that $\nn_0^{(I)}$ is a UFM with
	\[
		\mathcal{A}(\nn_0^{(I)}) = \mathcal{S}(\nn_0^{(I)}) = \mathcal{P}(\nn_0^{(I)}) = \{e_i : i \in I\}.
	\]
	Thus, $|\mathcal{S}(\nn_0^{(I)})| = |I|$, and so $\nn_0^{(I)}$ has infinitely many strong atoms provided that $I$ is an infinite set.
\end{ex}

The monoids we are most interested in are Krull monoids with torsion class groups. We proceed to discuss the set of strong atoms of the Hilbert monoid, which is perhaps the most elementary Krull monoid with torsion class group among those Krull monoids that are not UFMs.

\begin{ex} \label{ex:Hilbert monoid}
	Consider the multiplicative submonoid $H := 4 \nn_0 + 1$ of $\nn$, which is often called the \emph{Hilbert monoid}. It is clear that $H$ is atomic. In addition, it is well known and not difficult to argue that $\mathcal{A}(H) := P \cup Q$, where
	\[
		P := \big\{ p \in \pp : \ p \equiv 1 \! \! \! \pmod{4}\big\} \quad \text{ and } \quad Q := \big\{ p_1 p_2 : \ p_1, p_2 \in \pp \ \text{ and } \ p_1, p_2 \equiv 3 \! \! \! \pmod 4 \big\}.
	\]
	Clearly, every element in $P$ is prime in $H$ and, therefore, a strong atom of $H$. Observe, on the other hand, that for any distinct $p_1, p_2 \in \pp$ with $p_1, p_2 \equiv 3 \pmod 4$, the element $(p_1 p_2)^2$ has two factorizations in $H$, namely, $(p_1 p_2)^2$ and $p_1^2 p_2^2$. From this, we can deduce that none of the atoms in~$Q$ are primes in $H$, and so
	\[
		\mathcal{P}(H) = P.
	\]
	From the same statement, we can also deduce that the atoms in $Q$ that are the product of two distinct elements of $\pp$ cannot be strong atoms of $H$. On the other hand, for each $p \in \pp$ with $ p \equiv 3 \pmod 4$, any even power of $p$ has a unique factorization in $H$, and so $p^2$ is a strong atom of $H$. Therefore
	\[
		\mathcal{S}(H) = P \cup \big\{ p^2 : p \in \pp \ \text{ and } \ p \equiv 3 \! \! \! \pmod 4 \big\},
	\]
	and so $H$ contains infinitely many strong atoms that are not prime (observe that $H$ also contains infinitely many atoms that are not strong atoms). Finally, we notice that, as an immediate consequence of the Fundamental Theorem of Arithmetic, any two factorizations of the same element of $H$ must contain the same number of atoms  from $P$ (counting repetitions) and, therefore, the same number of atoms from $Q$ (counting repetitions). Hence $H$ is an HFM.
\end{ex}

\medskip
\subsection{Krull Monoids}

The monoid $M$ is called a \emph{Krull monoid} if $M$ has a divisor theory. It was proved by the second author that an integral domain is a Krull domain if and only if its multiplicative monoid is a Krull monoid \cite[Proposition]{uK89}. It is not hard to verify that $M$ is a Krull monoid if and only if $M_{\text{red}}$ is a Krull monoid. In light of this, we will tacitly assume that every Krull monoid we consider from now on is reduced. 

Let $M$ be a Krull monoid, and let $\tau \colon M \to F$ be a divisor theory of~$M$ for some free commutative monoid~$F$. In this case, $F$ is uniquely determined by $M$ (up to canonical isomorphism), and the quotient monoid $F/\tau(M)$ is actually an abelian group, which is denoted by $\text{Cl}(M)$ and called the \emph{class group} of~$M$ (see \cite[Theorem~2.4.7]{GH06}). The basis elements of $F$ are called the \emph{prime divisors} of~$M$. There is an extensive literature about the arithmetic of Krull monoids (see~\cite{GZ19} and references therein).

Every UFM is a Krull monoid: indeed, UFMs are precisely the Krull monoids with trivial class groups \cite[Corollary~2.3.13 and Theorem~2.4.7]{GH06}. In particular, the free commutative monoids in Example~\ref{ex:free commutative monoid} are Krull monoids with trivial class groups. The Hilbert monoid, on the other hand, is a Krull monoid that is not a UFM.

\begin{ex} \label{ex:Hilbert monoid is Krull}
	Observe that the multiplicative submonoid $F := 2\nn_0 + 1$ of $\nn$ is a free commutative monoid with basis $\pp \setminus \{2\}$, and it contains the Hilbert monoid $H$ as a submonoid. Now notice that for all $b,c \in H$, the relations $b \mid_H c$ and $b \mid_F c$ are equivalent. In addition, each $d \in F \setminus H$ is the greatest common divisor in $F$ of $d p_1$ and $d p_2$ for any distinct primes $p_1, p_2 \in 4 \nn_0 + 3$ such that $p_1 \nmid d$ and $p_2 \nmid d$. Thus, the inclusion $H \hookrightarrow F$ is a divisor theory and, therefore, $H$ is a Krull monoid whose set of prime divisors is $\pp \setminus \{2\}$. We have seen in Example~\ref{ex:Hilbert monoid} that $H$ is not a UFM, and so $\text{Cl}(H)$ is not the trivial group: indeed, $\text{Cl}(H) = F/H \cong \zz/2\zz$.
\end{ex}

\medskip
\subsection{The Decay Theorem}
\label{sec:Decay Theorem}

We are still assuming that the monoid $M$ is a Krull monoid. If $M$ has torsion class group, then an element of $M$ is a strong atom if and only if it is a primary atom \cite[Proposition~7.1.5]{GH06}. The Decay Theorem, which is the most crucial tool in this paper, states that every nonunit element of~$M$ has a power that uniquely decays into strong atoms. This applies especially to powers of atoms. Of course, it follows directly from the definition of an atom that an atom itself cannot decay, but when squared, for instance, it may decay into several other atoms. This decay of atoms by exercising some power is at the heart of all phenomena of non-unique factorizations when we restrict to the class of Krull monoids with torsion class groups (see Example~ \ref{ex:HF KM} and, for rings of integers of algebraic number fields, Example~ \ref{ex:Dedekind zeta function}). A formal statement of the Decay Theorem goes as follows.

\begin{theorem} \cite[Theorem~1]{KZ91}
	Let $M$ be a Krull monoid with torsion class group. Then for every $x \in M^\bullet$ there exist $m(x) \in \nn$ and $x(a) \in \nn_0$ for each $a \in \mathcal{S}(M)$ such that 
	\begin{equation} \label{eq:DT representation}
		x^{m(x)} = \prod_{a \in \mathcal{S}(M)} a^{x(a)}.
	\end{equation}
	Also, if $m(x)$ is minimally chosen, then the finite set $\{a \in \mathcal{S}(M) : x(a) > 0 \}$ is uniquely determined.
\end{theorem}

The following corollary is an important consequence of the Decay Theorem, and it will be helpful later.

\begin{cor} \label{cor:submonoids generated by strong atoms in a KM are UFMs}
	Let $M$ be a Krull monoid with torsion class group. If $A$ is a subset of $\mathcal{S}(M)$, then the submonoid $\langle A \rangle$ of $M$ is a UFM with set of primes~$A$.
\end{cor}

\begin{proof}
	If $A$ is empty, then $\langle A \rangle$ is the trivial monoid, and there is nothing to prove. Assume, on the other hand, that $A$ is a nonempty subset of $\mathcal{S}(M)$, and set $S := \langle A \rangle$. It is clear that $\mathcal{A}(S) = A$. Suppose now that the equality $\prod_{a \in A} a^{x(a)} = \prod_{a \in A} a^{y(a)}$ holds, where $x(a), y(a) \in \nn_0$ for all $a \in A$ and the set $\{a \in A : x(a) > 0 \text{ or } y(a) >0 \}$ is finite. For each $a \in A$, set $z(a) := \min\{x(a), y(a)\}$, and then set
	\[
		x'(a) := x(a) - z(a) \quad \text{ and } \quad y'(a) := y(a) - z(a).
	\]
	After canceling the element $\prod_{a \in A} a^{z(a)}$, we obtain the equality $\prod_{a \in A_1} a^{x'(a)} = \prod_{a \in A_2} a^{y'(a)}$, where $A_1 := \{a \in A : x'(a) > 0\}$ and $A_2 := \{a \in A : y'(a) > 0\}$. Observe that $A_1$ and $A_2$ are disjoint. Now the uniqueness in the last statement of the Decay Theorem ensures that $A_1 = A_2$ and, therefore, both sets $A_1$ and $A_2$ must be empty. As a result, $x(a) = z(a) = y(a)$ for all $a \in A$, and we can conclude that every element of $S$ has a unique factorization. Hence $S$ is a UFM, and so the set of primes of $S$ is precisely its set of atoms, namely, $A$.
\end{proof}

With notation as in the Decay Theorem, the element $x \in M^\bullet$ can decay in different ways if we do not assume that $m(x)$ is minimally chosen, that is, there are different choices for the set of exponents $\{m(x), x(a) : a \in \mathcal{S}(M)\}$. However, the sets of exponents of any two decay representations of $x$ as in \eqref{eq:DT representation} are connected as the following lemma indicates.

\begin{lem} \label{lem:exponent proportionality in DM representations} 
	Let $M$ be a Krull monoid with torsion class group. For $x \in M^\bullet$, consider the following two representations as in~\eqref{eq:DT representation}:
	\begin{equation} \label{eq:two DT representations}
		x^{m(x)} = \prod_{a \in \mathcal{S}(M)} a^{x_m(a)} \quad \text{ and } \quad	x^{n(x)} = \prod_{a \in \mathcal{S}(M)} a^{x_n(a)}.
	\end{equation}
	Then the equality
	\begin{equation} \label{eq:proportionality of the DM representations}
		\frac{x_m(a)}{m(x)} = \frac{x_n(a)}{n(x)}
	\end{equation}
	holds for every $a \in \mathcal{S}(M)$. 
\end{lem}

\begin{proof}
	Let $S$ be the submonoid of $M$ generated by the set of strong atoms of $M$; that is, $S := \langle \mathcal{S}(M) \rangle$. It follows from Corollary~\ref{cor:submonoids generated by strong atoms in a KM are UFMs} that $S$ is a UFM with set of primes $\mathcal{S}(M)$. From the representations in~\eqref{eq:two DT representations}, we obtain the following:
	\begin{equation} \label{eq:aux product of strong atoms}
		\prod_{a \in \mathcal{S}(M)} a^{n(x) x_m(a)} = x^{m(x) n(x)} = \prod_{a \in \mathcal{S}(M)} a^{m(x) x_n(a)}.
	\end{equation}
	Since $S$ is a UFM with set of primes $\mathcal{S}(M)$ and the equalities~\eqref{eq:aux product of strong atoms} take place in $S$, we see that $n(x) x_m(a) = m(x) x_n(a)$ for all $a \in \mathcal{S}(M)$, which yields the equalities in~\eqref{eq:proportionality of the DM representations}.
\end{proof}

For the rest of this section, assume that the Krull monoid $M$ has torsion class group. For each $a \in \mathcal{S}(M)$, we now define a function $\lambda_a \colon M \to \qq_{\ge 0}$ as follows: for each $u \in \uu(M)$ set $\lambda_a(u) := 0$, and for each $x \in M \setminus \uu(M)$ set
\[
	\lambda_a(x) := \frac{x(a)}{m(x)},
\]
where $m(x)$ and $x(a)$ are the exponents of $x$ and $a$ in \eqref{eq:DT representation}, provided that the exponent $m(x)$ has been chosen minimally. By virtue of Lemma~\ref{lem:exponent proportionality in DM representations}, the equality $\lambda_a(x) = \frac{x(a)}{m(x)}$ holds for any of the representations in~\eqref{eq:DT representation}, even when $m(x)$ is not minimally chosen. Observe that for any $a,b \in \mathcal{S}(M)$, the equality $\lambda_a(b) = 1$ holds if $a=b$ while $\lambda_a(b) = 0$ otherwise; that is, the values $\lambda_a(b)$ are those of the Kronecker's delta on the set $\mathcal{S}(M)$. Moreover, for each $x \in M^\bullet$, we observe that
\begin{equation} \label{eq:support in terms of the DT decomposition}
	\{a \in \mathcal{S}(M) : \lambda_a(x) > 0\}
\end{equation} 
is a finite set, which immediately follows from the last statement of the Decay Theorem.

Now, we can use the functions $\lambda_a$ (for all $a \in \mathcal{S}(M)$) to define another monoid homomorphism on~$M$. We define $\delta \colon M \to \qq_{\ge 0}$ via the assignment
\[
	\delta(x) := \sum_{a \in \mathcal{S}(M)} \lambda_a(x)
\]
for each $x \in M$. Note that $\delta$ is a well-defined function because, for each $x \in M$, the set in~\eqref{eq:support in terms of the DT decomposition} is finite. Observe that $\delta(a) = 1$ for any $a \in \mathcal{S}(M)$. Moreover, we can see that when $x \in \langle \mathcal{S}(M) \rangle$ the magnitude $\delta(x)$ equals the length of the unique factorization of $x$ in the monoid $\langle \mathcal{S}(M) \rangle$ (see Corollary~\ref{cor:submonoids generated by strong atoms in a KM are UFMs}). Roughly speaking, for each $x \in M^\bullet$, the magnitude $\delta(x)$ equals the total number of strong atoms \emph{per capita} into which $x^{m(x)}$ decays: we call $\delta(x)$ the \emph{decay rate} of $x$. 

\begin{lem}
	Let $M$ be a Krull monoid with torsion class group. The following statements hold.
	\begin{enumerate}
		\item For each $a \in \mathcal{S}(M)$, the function $\lambda_a \colon M \to \qq_{\ge 0}$ is a monoid homomorphism:
		\begin{equation} \label{eq:lambda additive property}
			\lambda_a(xy) = \lambda_a(x) + \lambda_a(y) \quad \text{for all} \quad x,y \in M.
		\end{equation}
		
		\item The function $\delta \colon M \to \qq_{\ge 0}$ is a monoid homomorphism:
		\begin{equation} \label{eq:decay rate property}
			\delta(xy) = \delta(x) + \delta(y) \quad \text{for all} \quad x,y \in M.
		\end{equation}
	\end{enumerate}
\end{lem}

\begin{proof}
	(1) Fix $a \in \mathcal{S}(M)$, and let us argue that $\delta_a$ is a monoid homomorphism. To do so, fix $x,y \in M$, and let $x^{m(x)} = \prod_{a \in \mathcal{S}(M)} a^{x(a)}$ and $y^{m(y)} = \prod_{a \in \mathcal{S}(M)} a^{y(a)}$ be the respective representations of $x$ and $y$ given by the Decay Theorem (assuming the minimality of $m(x)$ and $m(y)$). From these two representations, we obtain not only that $\lambda_a(x) = \frac{x(a)}{m(x)}$ and $\lambda_a(y) = \frac{y(a)}{m(y)}$, but also the following:
	\[
		(xy)^{m(x)m(y)} = \prod_{a \in \mathcal{S}(M)} a^{x(a) m(y) + y(a)m(x)}.
	\]
	It follows from Lemma~\ref{lem:exponent proportionality in DM representations}, that the value of the function $\lambda_a$ at $xy$ does not depend on the representation of $xy$ given by the Decay Theorem. Therefore
	\[
		\lambda_a(xy) = \frac{x(a) m(y) + y(a) m(x)}{m(x)m(y)} = \frac{x(a)}{m(x)} + \frac{y(a)}{m(y)} = \lambda_a(x) + \lambda_a(y).
	\]
	
	(2) This is an immediate consequence of the definition of $\delta$ and the fact that $\lambda_a$ is a monoid homomorphism for every $a \in \mathcal{S}(M)$.
\end{proof}
\smallskip

We have seen in Corollary~\ref{cor:submonoids generated by strong atoms in a KM are UFMs} that the submonoid generated by the strong atoms of $M$ is a UFM. Furthermore, $M$ itself is a UFM if and only if $\mathcal{S}(M) = \mathcal{P}(M)$ (the fact that $M$ is a Krull monoid is required here: see Example~\ref{ex:strong atoms of rank-1 monoids}). The quotient $M/\langle \mathcal{S}(M) \rangle$ is an abelian group, which is called the \emph{inner class group} of $M$ (see \cite[Section~3]{GKL19} for more detail). For every element $x \in M^\bullet$, the order of the class of $x$ in $M/\langle \mathcal{S}(M) \rangle$ equals $m(x)$. Comparing the unique representations given on one hand by the Decay Theorem and on the other hand by the basis of the free commutative monoid $F$ in a divisor theory $M \to F$ of $M$, we obtain the following corollary.

\begin{cor}  \cite[Theorem~1]{gA20} \label{cor:consequences of the DT}
	Let $M$ be a Krull monoid with torsion class group, and let $M \to F$ be a divisor theory of $M$. Then the following statements hold.
	\begin{enumerate}
		\item $a \in \mathcal{S}(M)$ if and only if $a = p^{k(p)}$ for some prime divisor $p \in F$, where $k(p)$ is the order of the class of $p$ in $\emph{Cl}(M)$.
		\smallskip
		
		\item $\mathcal{S}(M) = \{ a \in \mathcal{A}(M) : a \text{ is primary} \}$.
	\end{enumerate}
\end{cor}

For the unique decay of elements into special atoms, see~\cite[Theorem~4]{uK90}, \cite[Theorem~1]{KZ91}, \cite[Theorem~3.9]{CK05}, \cite[Theorem~4]{CHK02}, \cite[Theorem~3.1]{CK12}. A proof of the Decay Theorem using divisor theory can be found in \cite[Theorem~1]{KZ91} and~\cite[Theorem~1]{gA20}, while a proof of the Decay Theorem using extraction theory can be found in~\cite[Corollary~8]{GKL23}. In~\cite[Theorem~1]{gA20}, it is shown that, for a Krull monoid $M$, part~(2) of Corollary~\ref{cor:consequences of the DT} is equivalent to the fact that $M$ has torsion class group.
\smallskip

We have seen in Example~\ref{ex:Hilbert monoid is Krull} that the Hilbert monoid is a Krull monoid with class group $\zz/2\zz$. In the following example, we use the Decay Theorem to shed some light upon the (strong) atoms of Krull monoids with class group $\zz/2\zz$.

\begin{ex} \label{ex:HF KM}
	Let $M$ be a Krull monoid such that $\text{Cl}(M) \cong \zz/2\zz$, and let $\tau \colon M \to F$ be a divisor theory of $M$. Since $M$ is reduced, the homomorphism $\tau$ is injective, and so we can identify $M$ with its isomorphic copy $\tau(M)$. As $M$ is a submonoid of a free commutative monoid, $M$ is torsion-free. It follows from \cite[Proposition~2]{KZ91} that $M$ is an HFM (but not a UFM). As~$M$ is a Krull monoid with torsion class group, every $x \in M^\bullet$ has a unique representation as in~\eqref{eq:DT representation}. Since $m(x)$ divides the order of $\text{Cl}(M)$, either $m(x) = 1$ or $m(x) = 2$.
	
	For $a \in \mathcal{S}(M)$, the fact that $\text{Cl}(M)$ has order $2$, in tandem with Corollary~\ref{cor:consequences of the DT}, guarantees that either $a \in \mathcal{P}(F)$ or $a = p^2$ for some $p \in \mathcal{P}(F) \setminus M$. Conversely, it is clear that
	\[
		\mathcal{P}(F) \cap M \subseteq \mathcal{P}(M) \subseteq \mathcal{S}(M).
	\]
	In addition, suppose that $p \in \mathcal{P}(F) \setminus M$ and $p^2 \in M$. We claim that $p^2 \in \mathcal{S}(M)$. To check this, write $(p^2)^n = a_1 \cdots a_\ell$ for some $a_1, \dots, a_\ell \in \mathcal{A}(M)$. Since $F$ is free on $\mathcal{P}(F)$, for each $i \in \ldb 1,\ell \rdb$ the element $a_i$ is a power of $p$. Moreover, if $a_i = p^{2j+1}$ for some $i \in \ldb 1,\ell \rdb$ and $j \in \nn_0$, then the relation $p^{2j} \mid_F a_i$ would imply that $p^{2j} \mid_M a_i$, which is not possible because $p \notin M$. Thus, for each $i \in \ldb 1, \ell \rdb$, the element $a_i$ is an even power of $p$, and so the fact that $a_i \in \mathcal{A}(M)$ ensures that $a_i = p^2$. Hence $p^2 \in \mathcal{S}(M)$, as desired. Therefore
	\[
		\mathcal{S}(M) = (\mathcal{P}(F) \cap M) \cup \big\{p^2 : p \in \mathcal{P}(F) \setminus M \text{ and } p^2 \in M\big\}.
	\]
	
	For $a \in \mathcal{A}(M)$, it follows from~\eqref{eq:DT representation} and the minimality of $m(a)$ that $a \in \mathcal{S}(M)$ if and only if $m(a) = 1$. On the other hand, from the fact that $M$ is an HFM we obtain that if $m(a) = 2$, then $a^2 = a_1 a_2$ for some $a_1, a_2 \in \mathcal{S}(M)$, and so the fact that $M$ is torsion-free and the minimality of $m(a)$ guarantee that $a_1 \neq a_2$. Conversely, if $a^2 = a_1 a_2$ for distinct $a_1, a_2 \in \mathcal{S}(M)$, then $a$ is not a strong atom and, therefore, $m(a) = 2$. When $a^2 = a_1 a_2$  for distinct $a_1, a_2 \in \mathcal{S}(M)$, we say that the atom~$a$ \emph{splits} into strong atoms.
	
	Finally, for $x \in M^\bullet$, it is clear that the equality $m(x) = 1$ holds if and only if $x$ factors into strong atoms. As a consequence, when $m(x) = 2$ any factorization of $x$ in $M$ must have an atom that is not a strong atom, that is, an atom that splits into strong atoms.
\end{ex}

\bigskip
\section{Riemann Zeta Functions and Euler's Product Formula}
\label{sec:Zeta function for KM and Euler Product}

In order to generalize the Riemann zeta function to any Krull monoid, we need a measure of the elements of such a monoid (which in case of $\nn$ is given in a natural way by the magnitude of numbers); this is what we call a scale. For a Krull monoid $M$ with torsion class group, a \emph{scale} on $M$ is a monoid homomorphism $\sigma \colon M \to \rr^\times$ such that $\sigma(a) > 1$ for every $a \in \mathcal{S}(M)$. Let $\sigma \colon M \to \rr^\times$ be a scale on~$M$. It follows immediately that $\sigma(x) > 1$ for every nonunit element $x \in \langle \mathcal{S}(M) \rangle$. Moreover, as~$M$ is a Krull monoid with torsion class group, it follows from the Decay Theorem that if $x \in M$ is a nonunit element, then $x^{m(x)}$ is a nonunit element of $\langle \mathcal{S}(M) \rangle$, and so $\sigma(x) > 1$. 

As the following lemma indicates, any real-valued function on $\mathcal{S}(M)$ whose image is contained in $\rr_{>1}$ extends to a scale.

\begin{lem}
	Let $M$ be a Krull monoid with torsion class group, and let $f \colon \mathcal{S}(M) \to \rr_{> 1}$ be an arbitrary function. Then there exists a scale $\sigma \colon M \to \rr^\times$ extending $f$, that is, $\sigma(a) = f(a)$ for all $a \in \mathcal{S}(M)$.
\end{lem}

\begin{proof}
	It follows from the Decay Theorem that every nonunit element $x \in M$ has a unique representation in the form $x^{m(x)} = \prod_{a \in \mathcal{S}(M)} a^{x(a)}$, and so we can set
	 \[
		 \sigma(x) := \prod_{a \in \mathcal{S}(M)} f(a)^{\lambda_a(x)}.
	 \]
	 In addition, we set $\sigma(u) := 1$ if $u$ is a unit. Thus, we obtain a map $\sigma \colon M \to \rr^\times$. From the fact that $\lambda_a(yz) = \lambda_a(y) + \lambda_a(z)$ for all $y,z \in M^\bullet$ (see \eqref{eq:lambda additive property}) we obtain that $\sigma$ is indeed a monoid homomorphism, and so a scale on $M$ because $\sigma(a) = f(a) > 1$ for all $a \in \mathcal{S}(M)$.
\end{proof}

The notion of a scale is the crucial ingredient we will use to provide a suitable generalization of the classical Riemann zeta function. However, in doing so certain infinite sums and products will show up, so we first need to handle some convergence issue. In order to consider the convergence of certain infinite sums and products that may not be indexed by~$\nn$, we will use the following convention. If~$I$ is an arbitrary set of indices and $\{b_i : i \in I\}$ is a set consisting of nonnegative real numbers, we set $\sum_{i \in I} b_i = 0$ if $I$ is empty. In addition, we say that the sum $\sum_{i \in I} b_i$ \emph{converges} provided that there exists $c \in \rr_{\ge 0}$ such that $\sum_{s \in S} b_s \le c$ for any finite subset $S$ of~$I$, in which case we call
\[
	\ell := \sup \Big\{ \sum_{s \in S} b_s : S \subseteq I \ \text{and} \ |S| < \infty \Big\}
\]
the \emph{limit} of $\sum_{i \in I} b_i$: we also say that $\sum_{i \in I} b_i$ \emph{converges to}~$\ell$ and write $\sum_{i \in I} b_i = \ell$. Now suppose that $b_i \ge 1$ for all $i \in I$. Then we set $\prod_{i \in I} b_i = 1$ if $I$ is empty and, in addition, we say that the product $\prod_{i \in I} b_i$ \emph{converges} provided that there exists $c' \in \rr_{>1}$ such that $\prod_{s \in S} b_s \le c'$ for every finite subset $S$ of $I$, in which case we call
\[
	\ell' := \sup \Big\{ \prod_{s \in S} b_s : S \subseteq I \ \text{and} \ |S| < \infty \Big\}
\]
the \emph{limit} of $\prod_{i \in I} b_i$, say that $\prod_{i \in I} b_i$ \emph{converges} to $\ell'$, and write  $\prod_{i \in I} b_i = \ell'$. If a sum/product does not converge according to the given definition, we say that it \emph{diverges}. Clearly, when $I$ is countable, the given definitions of convergence reduce to the corresponding classical notions of convergence for series and infinite products.

\begin{rem}
	The notion of convergence we have just proposed is convenient only for notational purposes in the sense that $\sum_{i \in I} b_i$ diverges provided that the set $\{i \in I : b_i > 0\}$ is uncountable. To argue this, let~$I$ be an uncountable set, and take $b_i \in \rr_{> 0}$ for every $i \in I$. After setting
	\[
		B_n := \Big\{ i \in I : b_i > \frac1n \Big\}
	\]
	for every $n \in \nn$, we see that the equality $I = \cup_{n \in \nn} B_n$ holds. Since $I$ is uncountable, we can take $k \in \nn$ such that $|B_k| = \infty$.  Thus, for each $\ell \in \nn$, we can take a subset $S$ of $B_k$ with $|S| = k \ell$, and observe that $\sum_{s \in S} s > \frac{|S|}k = \ell$. Thus, $\sum_{i \in I} b_i$ diverges.
\end{rem}

We are now in a position to generalize the classical Riemann zeta function from the multiplicative monoid $\nn$ to any Krull monoid with torsion class group. 

\begin{defn} \label{def:Rieamman zeta function}
	For a Krull monoid $M$ with torsion class group, which is endowed with a scale $\sigma \colon M \to \rr^\times$, we call
	\begin{equation} \label{eq:generalize Zeta functino}
		\zeta_M(\sigma) := \sum_{x \in \langle \mathcal{S}(M) \rangle}  \frac 1{\sigma(x)}
	\end{equation}
	the (\emph{generalized}) \emph{Riemann zeta function} of $M$ at $\sigma$.
\end{defn}
Depending on the scale $\sigma$, the summation may be a divergent sum. For any $s \in \rr_{> 1}$, observe that when we pick the scale $\sigma \colon \nn \to \rr^\times$ defined by the assignments $\sigma \colon n \mapsto n^s$, what we obtain in~\eqref{eq:generalize Zeta functino} is the classical (real) Riemann zeta function evaluated at~$s$.

Our next goal is to obtain a generalization of Euler's product formula for the Riemann zeta function of any Krull monoid with torsion class group. As an immediate consequence of the Decay Theorem, we can write
\[
	\zeta_M(\sigma) = \sum_{x \in M, m(x)=1} \frac{1}{\sigma(x)} = \sum_{x \in M/\sim} \frac{1}{\sigma(x)},
\]
where $\sim$ is the equivalence relation on $M$ defined by $x \sim y$ if $x^{m(x)} = y^{m(y)}$. Sometimes these expressions are easy to handle. Towards our next goal, we need the following lemma.

\begin{lem} \label{lem:Euler product for finite subsets}
	Let $M$ be a Krull monoid with torsion class group, and let $\sigma \colon M \to \rr^\times$ be a scale on~$M$. If $A$ is a finite subset of $\mathcal{S}(M)$, then
	\begin{equation} \label{eq:Euler product for finite subsets}
		\sum_{x \in \langle A \rangle} \frac{1}{\sigma(x)} = \prod_{a \in A} \frac{1}{1 - \frac{1}{\sigma(a)}}.
	\end{equation}
\end{lem}

\begin{proof}
	Set $k := |A|$, and let $a_1, \dots, a_k$ be the strong atoms in $A$. Because~$A$ consists of strong atoms, the submonoid $\langle A \rangle$ of $M$ is a UFM in light of Corollary~\ref{cor:submonoids generated by strong atoms in a KM are UFMs}. Therefore, for each $x \in \langle A \rangle$, there is a unique $k$-tuple $(n_1, \dots, n_k)$ of nonnegative integers such that $x = a_1^{n_1} \cdots a_k^{n_k}$ and, moreover, the assignments $x \mapsto (n_1, \dots, n_k)$ induces a bijection $\langle A \rangle \to \nn_0^k$. Thus, we can write the left-hand side of~\eqref{eq:Euler product for finite subsets} as follows:
	\[
		\sum_{x \in \langle A \rangle} \frac{1}{\sigma(x)} = \sum_{(n_1, \dots, n_k) \in \nn_0^k} \frac{1}{\sigma(a_1^{n_1} \cdots a_k^{n_k})} = \sum_{(n_1, \dots, n_k) \in \nn_0^k} \frac{1}{\sigma(a_1)^{n_1} \cdots \sigma(a_k)^{n_k}},
	\]
	which does not require $\sigma$ to be injective. On the other hand, since $\sigma$ is a scale on $M$, it follows that $\sigma(a) > 1$ for each $a \in A$. Hence we can write the right-hand side of~\eqref{eq:Euler product for finite subsets} as follows:
	\[
		\prod_{a \in A} \frac{1}{1 - \frac{1}{\sigma(a)}} = \prod_{i=1}^k \frac{1}{1 - \frac{1}{\sigma(a_i)}} = \prod_{i=1}^k \bigg( \sum_{n=0}^\infty \frac{1}{\sigma(a_i)^n} \bigg) = \sum_{(n_1, \dots, n_k) \in \nn_0^k} \frac{1}{\sigma(a_1)^{n_1} \cdots \sigma(a_k)^{n_k}}.
	\]
	As a consequence, the identity in~\eqref{eq:Euler product for finite subsets} holds.
\end{proof}
 
We are in a position to prove our main result.

\begin{theorem} \label{thm:Euler product for Krull monoids}
	Let $M$ be a Krull monoid with torsion class group, and let $\sigma \colon M \to \rr^\times$ be a scale on~$M$. Then the following statements hold.
	\begin{enumerate}
		\item The sum $\zeta_M(\sigma) = \sum_{x \in \langle \mathcal{S}(M) \rangle} \frac{1}{\sigma(x)}$ converges if and only if the sum $\sum_{a \in \mathcal{S}(M)} \frac{1}{\sigma(a)}$ does.
		\smallskip
		
		\item If $\zeta_M(\sigma) = \sum_{x \in \langle \mathcal{S}(M) \rangle} \frac{1}{\sigma(x)}$ converges, then the following generalization of Euler's product formula holds:
		\begin{equation} \label{eq:Euler product for Krull monoids}
			\sum_{x \in \langle \mathcal{S}(M) \rangle} \frac{1}{\sigma(x)} =  \prod_{a \in \mathcal{S}(M)} \frac{1}{1 - \frac{1}{\sigma(a)}}.
		\end{equation}
	\end{enumerate}
\end{theorem}

\begin{proof}
	(1) It follows directly from our definition of a convergent sum that if $J$ is a set of indices and $\{b_j : j \in J\}$ is a set consisting of nonnegative real numbers such that the sum $\sum_{j \in J} b_j$ converges, then $\sum_{i \in I} b_i$ converges for any subset $I$ of $J$. As a consequence, it is clear that $\sum_{a \in \mathcal{S}(M)} \frac{1}{\sigma(a)}$ converges provided that $\sum_{x \in \langle \mathcal{S}(M) \rangle} \frac{1}{\sigma(x)}$ converges. To establish the converse, we first need to argue the following claim.
	\medskip
	
	\noindent \textit{Claim:} For any subset $R$ of $[0,1)$, the convergence of the sum $\sum_{r \in R} r$ implies the convergence of the product $\prod_{r \in R} \frac{1}{1-r}$.
	\smallskip
	
	\noindent \textit{Proof of Claim:} Let $R$ be a subset of $[0,1)$ such that the sum $\sum_{r \in R} r$ converges. Then the set $S := \{r \in R : r \ge \frac 12\}$ must be finite. Since dropping finitely many terms from $R$ does not affect the convergence of either $\sum_{r \in R} r$ or $\prod_{r \in R} \frac{1}{1-r}$, after replacing $R$ by $R \setminus S$, we can assume that $r < \frac 12$ for all $r \in R$. Thus, for each $r \in R$, we see that $(1-r)(1+2r) = 1 + r(1-2r) \ge 1$ and, therefore, $\frac{1}{1-r} \le 1 + 2r \le e^{2r}$. As a result, for any finite subset $T$ of $R$, we obtain that
	\[
		\prod_{t \in T} \frac{1}{1-t} \le \prod_{t \in T} e^{2t} = e^{2 \sum_{t \in T} t}.
	\]
	Hence the product $\prod_{r \in R} \frac{1}{1-r}$ converges provided that the sum $\sum_{r \in R} r$ converges (observe that $\frac{1}{1-r} \ge 1$ for all $r \in R$), and our claim is established.
	\smallskip
	
	Now suppose that the sum $\sum_{a \in \mathcal{S}(M)} \frac{1}{\sigma(a)}$ converges. Since $\{ \frac{1}{\sigma(a)} : a \in \mathcal{S}(M)\}$ is a subset of $[0,1)$, it follows from the established claim that the product $\prod_{a \in \mathcal{S}(M)} \big( 1 - \frac{1}{\sigma(a)} \big)^{-1}$ also converges. As a consequence, we obtain that
	\begin{equation} \label{eq:Euler sum aux}
		\sup_{A \subseteq \mathcal{S}(M), |A| < \infty} \, \sum_{x \in \langle A \rangle} \frac{1}{\sigma(x)} = \sup_{A \subseteq \mathcal{S}(M), |A| < \infty} \, \prod_{a \in A} \frac{1}{1 - \frac{1}{\sigma(a)}} = \prod_{a \in \mathcal{S}(M)} \frac{1}{1 - \frac{1}{\sigma(a)}},
	\end{equation}
	where the first equality follows from Lemma~\ref{lem:Euler product for finite subsets}. On the other hand, observe that for any finite subset~$S$ of $\langle \mathcal{S}(M) \rangle$ there exists a finite subset $A_S$ of $\mathcal{S}(M)$ such that $S \subseteq \langle A_S \rangle$ and, therefore,
	\begin{equation} \label{eq:Euler sum aux 2}
		\sum_{x \in S} \frac{1}{\sigma(x)} \le \sum_{x \in \langle A_S \rangle} \frac{1}{\sigma(x)} \le \sup_{A \subseteq \mathcal{S}(M), |A| < \infty} \, \sum_{x \in \langle A \rangle} \frac{1}{\sigma(x)} = \prod_{a \in \mathcal{S}(M)} \frac{1}{1 - \frac{1}{\sigma(a)}},
	\end{equation}
	where the last equality is that established in~\eqref{eq:Euler sum aux}. Hence the sum $\sum_{x \in \langle \mathcal{S}(M) \rangle} \frac{1}{\sigma(x)}$ converges, as desired.
	\smallskip
	
	(2) It follows from~\eqref{eq:Euler sum aux 2} that
	\[
		\sum_{x \in \langle \mathcal{S}(M) \rangle} \frac{1}{\sigma(x)} = \sup_{S \subseteq \langle \mathcal{S}(M) \rangle, |S| < \infty} \, \sum_{x \in S} \frac{1}{\sigma(x)} \le \prod_{a \in \mathcal{S}(M)} \frac{1}{1 - \frac{1}{\sigma(a)}}. 
	\]
	On the other hand, it follows from~\eqref{eq:Euler sum aux} that
	\[
		\sum_{x \in \langle \mathcal{S}(M) \rangle} \frac{1}{\sigma(x)} \ge \sup_{A \subseteq \mathcal{S}(M), |A| < \infty} \, \sum_{x \in \langle A \rangle} \frac{1}{\sigma(x)} = \prod_{a \in \mathcal{S}(M)} \frac{1}{1 - \frac{1}{\sigma(a)}}.
	\]
	Thus, we conclude that the identity~\eqref{eq:Euler product for Krull monoids} holds.
\end{proof}

Euler's classical product formula offers a simple proof of the infinitude of primes. In the same vein, our more general Euler's product formula can be used to supply a criterion to know whether a Krull monoid with torsion class group has infinitely many strong atoms based on the existence of a divergent scale or on the behavior of the decay rate. The following corollary sheds some light upon this observation.

\begin{cor}  \label{cor:when there are infinitely many strong atoms}
	Let $M$ be a Krull monoid with torsion class group. Then the following conditions are equivalent.
	\begin{enumerate}
		\item[(a)] $\mathcal{S}(M)$ is infinite.
		\smallskip
		
		\item[(b)] There exists a sequence $(x_n)_{n \ge 1}$ of pairwise distinct elements of $\langle \mathcal{S}(M) \rangle$ and a constant $c > 1$ such that, for every $n \ge 2$, the unique factorization of $x_n$ in the UFM $\langle \mathcal{S}(M) \rangle$ has length at most $\log_c n$.
		\smallskip
		
		\item[(c)] There exists  $r \in (0,1)$ such that $\sum_{x \in \langle \mathcal{S}(M) \rangle} r^{\delta(x)}$ diverges, where $\delta(x)$ is the decay rate of $x$.
		\smallskip
		
		\item[(d)] There exists a scale $\sigma$ on $M$ such that the sum $\zeta_M(\sigma)$ is divergent.
		\smallskip
		
		\item[(e)] There exists a scale $\sigma$ on $M$ such that $\sum_{a \in \mathcal{S}(M)} \frac1{\sigma(a)}$ diverges.
	\end{enumerate}
\end{cor}

\begin{proof}
	(a) $\Rightarrow$ (b): Assume that $\mathcal{S}(M)$ is an infinite set, and take a sequence $(a_n)_{n \ge 1}$ of pairwise distinct strong atoms of~$M$. It is clear that, for each $n \in \nn$, the length of the unique factorization of~$a_n$ in the UFM $\langle \mathcal{S}(M) \rangle$ is $1$. Then we can take $c=2$. 
	\smallskip
	
	(b) $\Rightarrow$ (c): Suppose there exist a sequence $(x_n)_{n \ge 1}$ and a constant $c$ as described in the statement of condition~(b). Set $r := \frac1c$ and, for each $n \in \nn$, let $\ell_n$ denote the length of the unique factorization of $x_n$ in $\langle \mathcal{S}(M) \rangle$. For each nonzero $x \in \langle \mathcal{S}(M) \rangle$, it follows from the definition of the decay rate that
	\[
		\delta(x) = \sum_{a \in \mathcal{S}(M)} \frac{x(a)}{m(x)} = \sum_{a \in \mathcal{S}(M)} x(a), 
	\]
	which is the length of the unique factorization of $x$ in $\langle \mathcal{S}(M) \rangle$. Thus, for each $n \in \nn_{\ge 2}$, it follows that $c^{\delta(x_n)} = c^{\ell_n} \le c^{\log_c n} = n$, and so $r^{\delta(x_n)} = c^{-\delta(x_n)} \ge \frac1n$. This implies that
	\[
		\sum_{x \in \langle \mathcal{S}(M) \rangle} r^{\delta(x)} \ge \sum_{n \ge 2} r^{\delta(x_n)} \ge \sum_{n \ge 2} \frac1n,
	\]
	whence $\sum_{x \in \langle \mathcal{S}(M) \rangle} r^{\delta(x)}$ diverges.
	\smallskip
	
	(c) $\Rightarrow$ (d): Let $r$ be as described in condition~(c). Now for each $x \in M$, set $\sigma(x) := r^{-\delta(x)}$. For any $y,z \in M$, the identity $\delta(yz) =  \delta(y) + \delta(z)$ (see \eqref{eq:decay rate property}) implies that $r^{-\delta(yz)} = r^{-\delta(y)} r^{-\delta(z)}$, and so~$\sigma$ is a monoid homomorphism. Now the fact that $r \in (0,1)$ implies that $\sigma(a) > 1$ for all $a \in \mathcal{S}(M)$, whence~$\sigma$ is a scale on~$M$. In addition,
	\[
		\zeta_M(\sigma) = \sum_{x \in \langle \mathcal{S}(M) \rangle} \frac1{\sigma(x)} = \sum_{x \in \langle \mathcal{S}(M) \rangle} r^{\delta(x)},
	\]
	and so $\zeta_M(\sigma)$ diverges.
	\smallskip
	
	(d) $\Rightarrow$ (e): This follows from part~(1) of Theorem~\ref{thm:Euler product for Krull monoids}.
	\smallskip
	
	(e) $\Rightarrow$ (a): This follows immediately. 
\end{proof}

\begin{rem}
	It should be noticed that many of the proofs of the infinitude of primes in the classical setting, including the notorious proof by Euclid, use both the addition and the multiplication of $\nn$. However, the more general criterion given by Corollary~\ref{cor:when there are infinitely many strong atoms} is only based on the sole operation of a Krull monoid (only the multiplication in the special case of $\nn$).
\end{rem}

With notation as in Corollary~\ref{cor:when there are infinitely many strong atoms}, the scale defined by the assignments $\sigma \colon n \mapsto n$ can be chosen (as Euler did in his purely multiplicative proof) to argue that $\nn$ contains infinitely many primes. A monoid $M_0$ isomorphic to $\nn$ also possesses infinitely many primes. From the same scale on $\nn$, we can obtain a scale on~$M_0$ yielding the infinitude of primes of $M$ via Corollary~\ref{cor:when there are infinitely many strong atoms} as follows. Let $\tau \colon M \to M'$ be a monoid isomorphism, and let $\sigma$ be a scale on $M$. Then the map $\sigma' :=  \sigma \circ \tau^{-1}$ is a scale on $M'$, and since $\sum_{y \in \mathcal{S}(M')} \frac1{\sigma'(y)} = \sum_{x \in \mathcal{S}(M)} \frac1{\sigma(x)}$, the scale $\sigma'$ satisfies condition~(e) in Corollary~\ref{cor:when there are infinitely many strong atoms} if and only if $\sigma$ does. For instance, it is clear that $M_0 := \{n^2 : n \in \nn\}$ is a multiplicative monoid isomorphic to $\nn$, under the isomorphism $\tau \colon \nn \to M_0$ given by the assignments $\tau \colon n \mapsto n^2$. For the scale $\sigma \colon n \mapsto  n$ on $\nn$, we obtain the scale $\sigma' := \sigma \circ \tau^{-1}$ on $M_0$, which is given by the assignments $\sigma' \colon n \mapsto \sqrt{n}$. This shows again the flexibility of putting the Riemann zeta function in terms of a variable scale $\sigma$.
\smallskip

In the next corollary, we connect the unique factorization property with Euler's product formula in the context of Krull monoids endowed with convergent scales.

\begin{cor} \label{cor:Riemann zeta function and Euler's product for UFMs}
	Let $M$ be a Krull monoid with torsion class group, and let $\sigma$ be a scale on~$M$ such that $\sum_{a \in \mathcal{S}(M)} \frac1{\sigma(a)}$ converges. Then the following statements hold.
	\begin{enumerate}
		\item The monoid $M$ is a UFM if and only if Euler's classical product formula holds; that is,
		\begin{equation} \label{eq:classic Euler's formula}
			\sum_{x \in M} \frac1{\sigma(x)} = \prod_{p \in \mathcal{P}(M)} \frac{1}{1 - \frac{1}{\sigma(p)}}.
		\end{equation}
		\smallskip
		
		\item If $M$ is a UFM, then $M$ has infinitely many primes if and only if the anti-geometric series $\sum_{x \in M} r^{\delta(x)}$ diverges for some $r \in (0,1)$.
	\end{enumerate}
\end{cor}

\begin{proof}
	(1) If $M$ is a UFM, then $\mathcal{S}(M) = \mathcal{P}(M)$ and $M = \langle \mathcal{S}(M) \rangle$, whence Euler's product formula~\eqref{eq:classic Euler's formula} follows from Theorem~\ref{thm:Euler product for Krull monoids}. Conversely, suppose that Euler's product formula~\eqref{eq:classic Euler's formula} does hold. Since $M$ is a Krull monoid with torsion class group, it follows from Theorem~\ref{thm:Euler product for Krull monoids} that
	\[
		\sum_{x \in \langle \mathcal{S}(M) \rangle} \frac1{\sigma(x)} = \prod_{a \in \mathcal{S}(M)} \frac1{1 - \frac1{\sigma(a)}}.
	\]
	As $\mathcal{P}(M) \subseteq \mathcal{S}(M)$,
	\begin{equation} \label{eq:auxi expression}
		\sum_{x \in M} \frac1{\sigma(x)} = \prod_{p \in \mathcal{P}(M)} \frac1{1 - \frac1{\sigma(p)}} \le \prod_{a \in \mathcal{S}(M)} \frac1{1 - \frac1{\sigma(a)}} = \sum_{x \in \langle \mathcal{S}(M) \rangle} \frac1{\sigma(x)}.
	\end{equation}
	Now the inclusion $\langle \mathcal{S}(M) \rangle \subseteq M$ guarantees that the inequality in~\eqref{eq:auxi expression} is indeed an equality. Therefore, from the fact that $\big( 1 - \frac1{\sigma(a)}\big)^{-1} > 1$ for all $a \in \mathcal{S}(M)$, we obtain that $\mathcal{P}(M) = \mathcal{S}(M)$. Since $\sigma(x) > 1$ for any nonunit element $x \in M$, the equality $\sum_{x \in M} \frac1{\sigma(x)}  = \sum_{x \in \langle \mathcal{S}(M) \rangle} \frac1{\sigma(x)}$ implies that $M = \langle \mathcal{S}(M) \rangle$. Thus, $M = \langle \mathcal{P}(M) \rangle$, which implies that $M$ is a UFM (cf. Example~\ref{ex:strong atom of a rank-3 monoid}).
	\smallskip
	
	(2) If $M$ is a UFM, then both equalities $\mathcal{P}(M) = \mathcal{S}(M)$ and $M = \langle \mathcal{S}(M) \rangle$ hold, and so the desired condition follows from the equivalence (a) $\Leftrightarrow$ (c) in Corollary~\ref{cor:when there are infinitely many strong atoms}.
\end{proof}

\begin{rem}
	Corollary~\ref{cor:Riemann zeta function and Euler's product for UFMs} applies, in particular, to the multiplicative monoid $\nn$, illustrating that the Fundamental Theorem of Arithmetic is equivalent to Euler's product formula for $\nn$. On the other hand, the Hilbert monoid $H = 4 \nn_0 + 1$ is not a UFM (see Example~\ref{ex:Hilbert monoid}), and so Corollary~\ref{cor:Riemann zeta function and Euler's product for UFMs} applied to $H$ with the scale $\sigma(n) = n^s$ for some fixed $s \in \rr_{>1}$ shows that Euler's product formula cannot hold for $H$. See also \cite[page 169]{hC80}.
\end{rem}

It is interesting to look at the multiplicative monoid $\nn$ from the perspective of Theorem~\ref{thm:Euler product for Krull monoids} and Corollaries~\ref{cor:when there are infinitely many strong atoms} and~\ref{cor:Riemann zeta function and Euler's product for UFMs}.

\begin{ex} \label{ex:classical setting}
	The multiplicative monoid $\nn$ is a UFM and, therefore, a Krull monoid with trivial class group.
	\begin{enumerate}
		\item For any $s \in \rr_{> 1}$, the map $\sigma_s \colon \nn \to \rr^\times$ defined by $\sigma_s(n) = n^s$ is a scale on $\nn$ and the Riemann zeta function $\zeta_\nn \colon \sigma_s \mapsto \zeta_\nn(\sigma_s)$ yields the classical (real) Riemann zeta function. In addition, with this choice of scales, the identity~\eqref{eq:Euler product for Krull monoids} gives us back Euler's classical product formula~\eqref{eq:Euler's product formula}. In particular, the choice of the scale given by the assignments $\sigma \colon n \mapsto n$ in Corollary~\ref{cor:when there are infinitely many strong atoms} gives us the existence of infinitely many primes, which is basically Euler's classical proof.
		\smallskip
		
		\item Another possible scale on $\nn$ is the map $\sigma_c \colon \nn \to \rr^\times$ defined by $\sigma_c(n) = c^{\delta(n)}$ for some fixed constant $c \in (1,2)$. Since $\nn$ has infinite set of strong atoms, namely $\pp$, according to Corollary~\ref{cor:when there are infinitely many strong atoms} there exists $r \in (0,1)$ such that $\sum_{n \in \nn} r^{\delta(n)}$ diverges. Let us verify that $r := \frac1c$ works. After fixing $n \in \nn$ and writing $n = \prod_{p \in \pp} p^{n(p)}$, it follows that
		\[
			\log_c n = \sum_{p \in \pp} n(p) \log_c p \ge \sum_{p \in \pp} n(p) = \delta(n)
		\]
		and, therefore, $n \ge c^{\delta(n)}$. This implies that $\sum_{n \in \nn} r^{\delta(n)}$ diverges, as desired. In contrast to the divergence of the harmonic series (as used in Euler's proof), the divergence of the series $\sum_{n \in \nn} r^{\delta(n)}$ explains the infinitude of primes in the sense that the frequency of prime factors is shrinking with $\log_c n$.
		\smallskip
		
		\item Since each scale on $\nn$ is determined by its values on $\pp$, in order to define a scale on $\nn$, we can start by fixing any function $f \colon \pp \to \rr^\times$, and then consider the function $\sigma_f \colon \nn \to \rr^\times$ defined as follows: for every $n = \prod_{p \in \pp} p^{n(p)}$, set
		\[
			\sigma_f(n) := \prod_{p \in \pp} p^{f(p) n(p)}.
		\]
		One can easily verify that $\sigma_f$ is a scale on $\nn$. Let us call this the \emph{Riemann zeta function weighted by} $f$. Assuming that the series $\sum_{p \in \pp} \frac{1}{f(p)}$ converges, we obtain from part~(1) of Theorem~\ref{thm:Euler product for Krull monoids} that
		\[
			\sum_{n \in \nn} \frac{1}{\sigma_f(n)} = \prod_{p \in \pp} \frac{1}{1 - \frac{1}{p^{f(p)}}}.
		\]
		\smallskip
		
		\item Now let $p_k$ denote the $k$-th prime, and consider the scale $\sigma$ on $\nn$ determined by the assignments $p_k \mapsto 4k^2$ for every $k \in \nn$. Since $\sum_{n=1}^\infty \frac{1}{\sigma(p_n)} = \sum_{n=1}^\infty \frac{1}{4n^2}$ converges, it follows from Theorem~\ref{thm:Euler product for Krull monoids} that the series $\sum_{n \in \nn} \frac{1}{\sigma(n)}$ also converges, and it also follows that
		\[
			 \sum_{n \in \nn} \frac{1}{\sigma(n)} = \prod_{n=1}^\infty \frac{1}{1 - \frac{1}{\sigma(p_n)}} = \prod_{n=1}^\infty \frac{4n^2}{4n^2 - 1} = \frac{\pi}{2}.
		\]
		The last equality is well known: it was established by J. Wallis back in 1656.
	\end{enumerate}
\end{ex}
\smallskip

Free commutative monoids are one of the simplest Krull monoids with torsion class group, and so Theorem~\ref{thm:Euler product for Krull monoids} applies to them.

\begin{ex}
	Let $I$ be a nonempty set, and let $M$ denote the free commutative monoid $\nn_0^{(I)}$ on $I$. As mentioned in Example~\ref{ex:free commutative monoid}, the monoid $M$ is a UFM with $\mathcal{P}(M) = \{e_i : i \in I\}$, the basis of $M$. Thus, for each $f \in M$ we can write $f = \sum_{i \in I} k_i e_i$ uniquely with $k_i = 0$ for all but finitely many $i \in I$. For each $i \in I$, pick $r_i \in \rr_{> 1}$ such that $\sum_{i \in I} \frac 1{r_i}$ converges, and consider the function $\sigma \colon M \to \rr^\times$ defined by $\sigma(f) := \prod_{i \in I} r_i^{k_i}$. We can readily check that $\sigma$ is a scale on $M$. Since $\sum_{i \in I} \frac1{\sigma(e_i)} = \sum_{i \in I} \frac1{r_i}$ converges, it follows from Theorem~\ref{thm:Euler product for Krull monoids} that $\zeta_M(\sigma) = \sum_{f \in M} \frac{1}{\sigma(f)}$ also converges and
	\[
		\sum_{f \in M} \frac{1}{\sigma(f)} =  \prod_{i \in I} \frac{1}{1 - \frac{1}{r_i}}.
	\]
\end{ex}

As the following example illustrates, Theorem~\ref{thm:Euler product for Krull monoids} applies specially well to Krull monoids with class group $\zz/2\zz$.

\begin{ex}
	Let $M$ be a Krull monoid with $\text{Cl}(M) \cong \zz/2\zz$, and let $\tau \colon M \to F$ be a divisor theory of $M$. We have seen in Example~\ref{ex:HF KM} that $M$ is an HFM with $\mathcal{S}(M) = \mathcal{S}_1 \sqcup \mathcal{S}_2$, where 
	\[
		\mathcal{S}_1 = \mathcal{P}(F) \cap M \quad \text{ and } \quad \mathcal{S}_2 =  \{p^2 : p \in \mathcal{P}(F) \setminus M \text{ and } p^2 \in M\}.
	\]
	Then for each scale $\sigma$ on $M$, the sum $\zeta_M(\sigma) = \sum_{x \in M} \frac{1}{\sigma(x)}$ converges if and only if both sums $\sum_{a \in \mathcal{S}_1} \frac{1}{\sigma(a)}$ \, and \, $\sum_{a \in \mathcal{S}_2} \frac{1}{\sigma(a)}$ converge, in which case, the following generalized version of Euler's product formula holds:
	\[
		\zeta_M(\sigma) = \sum_{x \in \langle \mathcal{S}(M) \rangle} \frac{1}{\sigma(x)} = \prod_{a \in \mathcal{S}_1 \sqcup \mathcal{S}_2} \frac1{1 - \frac1{\sigma(a)}}.
	\]
\end{ex}
\medskip

The nonzero ideals of the ring of integers of an algebraic number field form a multiplicative monoid, which is known to be a Krull monoid with torsion class group. In this setting, the generalized Riemann zeta function (as introduced in this paper) will specialize to the well-known Dedekind zeta function, yielding Euler's classical product representation for algebraic number fields. We conclude this paper taking a look at this situation.

\begin{ex} \label{ex:Dedekind zeta function}
	Let $\mathcal{O}_K$ be the ring of integers of an algebraic number field $K$. Also, let $M$ be the multiplicative monoid consisting of all nonzero principal ideals of $\mathcal{O}_K$, and let $S$ be the multiplicative monoid of nonzero ideals of $\mathcal{O}_K$. It is well known that the inclusion $M \hookrightarrow S$ is a divisor theory for $M$ with the prime ideals as prime divisors and also that the class group $\text{Cl}(M) = S/M$ is finite. Therefore $M$ is a Krull monoid with torsion class group. Let $\sigma$ be a scale on $M$ such that $\sum_{a \in \mathcal{S}(M)} \frac1{\sigma(a)}$ converges. It follows from Theorem~\ref{thm:Euler product for Krull monoids} that
	\begin{equation} \label{eq:Euler formula for ring of integers}
		\sum_{x \in \langle \mathcal{S}(M) \rangle} \frac{1}{\sigma(x)} =  \prod_{a \in \mathcal{S}(M)} \frac{1}{1 - \frac{1}{\sigma(a)}}.
	\end{equation}
	We proceed to argue that when the scale $\sigma$ is given by the norm of ideals, then Euler's product formula~\eqref{eq:Euler formula for ring of integers} yields Euler's classical product representation of the Dedekind zeta function of the algebraic number field $K$ (see \cite[Section~42]{eH70}). For this, let $N \colon S \to \nn_0$ be the (absolute) norm function; that is, $N(I) := |\mathcal{O}_K/I|$ for each $I \in S$. According to Corollary~\ref{cor:consequences of the DT}, $a \in \mathcal{S}(M)$ if and only if $a = P_a^{k(P_a)}$ for some prime ideal $P_a \in S$, where $k(P_a)$ is the order of the class of $P_a$ in the class group $S/M$. From this, we obtain that the map $\varphi \colon \mathcal{S}(M) \to \mathcal{P}(S)$ defined via the assignments $\varphi(a) = P_a$ is a bijection that extends uniquely to a monoid isomorphism $\langle \mathcal{S}(M) \rangle \to S$, which we also denote by $\varphi$ (more information about the connection between strong atoms and prime ideals in the setting of integral domains can be found in \cite[Section~2]{CK12}). For each $s \in \rr_{> 1}$, we can use the Decay Theorem to define the function
	\[
		\sigma_s \colon M \to \rr^\times \quad \text{ by } \quad \sigma_s(x) = N\big(\varphi\big(x^{m(x)}\big)\big)^{\frac{s}{m(x)}}
	\]
	for each $x \in M$ (observe that $\sigma_s(x) = N(\varphi(x))^s$ when $x \in \langle \mathcal{S}(M) \rangle$). Using the fact that $N$ is a multiplicative function and $\varphi$ is a monoid homomorphism, one can readily check that $\sigma_s$ is a scale on~$M$. Now suppose that the sum $\sum_{P \in \mathcal{P}(S)} \frac1{N(P)^s}$ converges. In this case, the sum $\sum_{x \in \langle \mathcal{S}(M) \rangle} \frac1{\sigma_s(x)}$ also converges and
	\[
		\zeta_M(\sigma_s) = \sum_{x \in \langle \mathcal{S}(M) \rangle} \frac1{\sigma_s(x)} =  \sum_{x \in \langle \mathcal{S}(M) \rangle} \frac1{N(\varphi(x))^s} = \sum_{I \in S} \frac1{N(I)^s}.
	\]
	Thus, for each $s \in \rr_{> 1}$, the Riemann zeta function $\zeta_M$ of $M$ evaluated at the scale $\sigma_s$ coincides with the Dedekind zeta function $\zeta_K$ evaluated at $s$. As a consequence, the general version of Euler's product formula~\eqref{eq:Euler product for Krull monoids} for Krull monoids specializes to Euler's classical product representation for the Dedekind zeta function
	\[
		\zeta_K(s) = \prod_{P \in \mathcal{P}(S)} \frac{1}{1 - \frac1{N(P)^s}}.
	\]
\end{ex}

\begin{rem}
	Theorem~\ref{thm:Euler product for Krull monoids} holds for any scale on $M$ and it is ``intrinsic" in the sense that it is only based on the operation of $M$ (the notion of a strong atom of $M$ is solely based on the operation of~$M$). In this direction, we conclude by observing that the Dedekind zeta function and its Euler's product representation are based on the notion of an ideal, which invokes not only the multiplication of $\mathcal{O}_K$ but also its addition.
\end{rem}

\bigskip
\section*{Acknowledgments}

The authors are grateful to an anonymous referee for useful comments and suggestions that helped improve an earlier version of this paper. While working on this paper, the first author was kindly supported by the NSF award DMS-2213323.
\bigskip

\bigskip

\end{document}